\documentclass[a4paper,11pt]{article}

\usepackage{amssymb, amsmath, a4wide,color,enumerate}
\usepackage{mathrsfs}

\newcommand{\ee}{\mathbb E}
\newcommand{\pp}{\mathbb P}
\newcommand{\rr}{\mathbb R}

\newcommand{\bz}{\boldsymbol z}
\newcommand{\bG}{\boldsymbol G}

\newtheorem{theorem}{Theorem}
\newtheorem{lemma}{Lemma}

\newtheorem{cor}{Corollary}
\newtheorem{remark}{Remark}

\newcommand{\proofname}{\noindent {\bf Proof}}
\newenvironment{proof}[1][\proofname]{\par \normalfont \trivlist
 \item[\hskip\labelsep\itshape #1]\ignorespaces
}{
 \hspace*{\fill}$\Box$ \endtrivlist
}
\definecolor{darkgreen}{rgb}{0,0.4,0}

\newcommand{\vc}[1]{{\boldsymbol #1}}

\newcommand{\sr}[1]{{\cal #1}}
\newcommand{\dd}[1]{\mathbb{#1}}

\newcommand{\eq}[1]{(\ref{eq:#1})}
\newcommand{\lem}[1]{Lemma~\ref{lem:#1}}
\newcommand{\cort}[1]{\ref{cor:#1}}
\newcommand{\thr}[1]{Theorem~\ref{thr:#1}}

\newcommand{\rem}[1]{Remark~\ref{rem:#1}}

\newcommand{\app}[1]{Appendix~\ref{app:#1}}
\newcommand{\sectn}[1]{Section~\ref{sect:#1}}

\begin{document}

\title{Queue-length balance equations in multiclass multiserver queues and their generalizations}

\author{Marko A.A. Boon\thanks{Department of Mathematics and Computer Science,
Eindhoven University of Technology, the Netherlands. \tt{m.a.a.boon@tue.nl}} \and
Onno J.\ Boxma \thanks{Department of Mathematics and Computer Science,
Eindhoven University of Technology, the Netherlands. \tt{o.j.boxma@tue.nl}}
\thanks{Partly funded by the NWO Gravity Project NETWORKS, Grant Number 024.002.003.}
\and Offer Kella\thanks{Department of Statistics, The Hebrew University of Jerusalem, Jerusalem 9190501, Israel. \tt{offer.kella@gmail.com}}
\thanks{Supported in part by grant 1462/13 from the Israel Science Foundation and the Vigevani Chair in Statistics.}
\and Masakiyo Miyazawa\thanks{Department of Information Sciences,
    Tokyo University of Science,
    Noda City, Chiba 278, Japan. \tt{miyazawa@rs.tus.ac.jp}}
\thanks{Supported in part by JSPS KAKENHI Grant Number 16H027860001.}}

\date{April 04, 2017}

\maketitle
\begin{abstract}
A classical result for the steady-state queue-length distribution of single-class queueing systems is the following:
the distribution of the queue length just before an arrival epoch equals
the distribution of the queue length just after a departure epoch.
The constraint for this result to be valid is that arrivals, and also service completions,
with probability one occur individually, i.e., not in batches.
We show that it is easy to write down somewhat similar balance equations for {\em multidimensional} queue-length processes for a quite general
network of multiclass multiserver queues.
We formally derive those balance equations under a general framework. They are called distributional relationships, and are obtained for any external arrival process and state dependent routing as long as certain stationarity conditions are satisfied and external arrivals and service completions do not simultaneously occur.
We demonstrate the use of these balance equations, in combination with PASTA, by (i) providing very simple derivations of some known results for polling systems, and (ii) obtaining new results for some queueing systems with priorities.
We also extend the distributional relationships for a non-stationary framework.

\end{abstract}

\section{Introduction}
A classical result for the steady-state queue-length distribution of single-class queueing systems is the following:
the distribution of the queue length just before an arrival epoch
equals the distribution of the queue length just after a departure epoch.
The constraint for this result to be valid is that, with probability one,
arrivals, and also service completions, occur individually, i.e., not in batches.
The result then follows by a simple level-crossing argument: in steady state, the event that a customer
arrives to find $j$ customers present occurs just as often as the event
that a customer leaves $j$ customers behind, for all $j=0,1,\dots$.
See \cite{Cooper}, pp. 154-156, for a formal statement and proof (due to P.J. Burke, unpublished) of this result.

At first sight this level-crossing argument breaks down
in higher dimensions, for example in the case of multiple customer classes.
Indeed, with $\vc{x }\ge \vc{0}$ and $\vc{e}_k$ being a unit vector with $1$ in the $k$th coordinate and zero elsewhere,
an $m$-dimensional process can leave state $\vc{x}$ because of an arrival of a customer of type $i$,
and enter that state from
state $\vc{x}+\vc{e}_k$ because of a departure of a customer of {\em another} type $k$.
However, we shall argue that it is easy to write down a more global balance equation for multidimensional queue length processes
for a large class of queues and queueing networks -- also when service times are not exponentially distributed,
and even when arrivals may occur in batches.
We shall explore that fact to obtain a simple relation between the steady-state joint queue-length distribution
at arrival epochs (which under various circumstances is equal to the time average distribution) and at service completion epochs.
Once one has a relation between the probability generating function (PGF) at arbitrary epochs and at service completion epochs,
one can find the former when one has the latter.
The latter results are indeed known in
an $M/G/1$ setting, where it is natural to look at departure epochs.
This will yield both new results (for multiclass queueing models with fixed priorities and
for the longer-queue model),
as well as new and simple derivations of known results for, e.g., polling models.

The research for the present paper was initially motivated by the desire to provide an intuitive
explanation of a result in \cite{BKK} regarding the steady-state joint queue-length distribution in a large class of polling models.
That distribution turned out to have a remarkably simple relation with a weighted sum of the joint queue length
distributions at departure epochs of customers from each of the queues.
In Section~\ref{sec:2} we provide such an explanation.
Although balance equations are intuitively appealing, their mathematical verification may require a large amount of work. This motivates us to derive
distributional relationships for queue lengths in a unified way using a general tool. The so called rate conservation law is such a tool as demonstrated in \cite{Miya2010} (also see \cite{BaccBrem2003,Miya1994}). This method is applicable to a general model, but requires Palm distributions, which may not be easy to understand.
In \sectn{formal} of this paper we take another approach, based on a time evolution of a sample path. This approach is parallel to the rate conservation law, but does not require Palm distributions, which are replaced by sample averages. We apply it to a general model, and derive a distributional relationship among different embedded epochs.
In \sectn{non stationary} a non-stationary version of the distributional relationship is derived with some error term, which vanishes as time goes to infinity.
Our main result, viz.
Theorem~\ref{thr:stationary relation 1},
as well as the non-stationary results, are novel to the best of our knowledge.
\\

\noindent
{\em Literature review.}
H\'ebuterne \cite{Hebuterne} provides a generalization of the above-mentioned classical result of Burke
in two directions: he allows (i) batch arrivals, with batches of random size,
and (ii) batch services, with batches of fixed size.
He also points out that emptying the queue up to $N$ customers is beyond the scope of the analysis,
because then the batch sizes  are not independent of the system state.
Fakinos \cite{Fakinos} manages to treat a quite general group-arrival group-departure queue.
He treats the batch size problem by assuming that customers within a departing group  are randomly ordered,
and that they leave the system according to their order.
Papaconstantinou and Bertsimas \cite{PapaBertsimas} generalize Burke's result to the multiserver $E_k/G/s$ queue.
Kim \cite{Kim} combines the features of batch arrivals, batch services and multiple servers,
also allowing multiple customer classes. he does not explicitly address the issue of customers in a departing group being randomly ordered.
H\'ebuterne and Rosenberg \cite{HebuterneRosenberg} focus on the $G/G/1$ queue with batch services and finite capacity.
Takine has obtained several relations between  queue lengths at random instants and at departure instants;
see in particular the very general Theorem 1 in \cite{Takine}, for a single server queue with multiple Markovian arrival streams -- an extension
of Markovian arrival processes to (possibly correlated) multiple arrival streams.
\\

\noindent
{\em Organization of the paper.}
Section~\ref{sec:2} provides a short proof of a result in \cite{BKK} by using a multi-dimensional queue-length balance argument.
\sectn{formal} derives the distributional relationship for an open queueing network under a very general setting in \thr{stationary relation 1}.
Extensions
to the non-stationary case are discussed in \sectn{non stationary}.
Some applications are presented in Section~\ref{sec:3}.
\sectn{concluding} contains concluding remarks.

\section{A balance equation for a class of polling models}
\label{sec:2}
In this section we provide a simple relation between the steady-state joint queue-length distribution at arbitrary epochs and at departure epochs for polling models. This relation, which is derived by introducing a multi-dimensional queue-length balance argument, is used to provide a short, but somewhat intuitive derivation of Theorem 1 of \cite{BKK}.
In the next section we shall extend that balance equation in a very general setting, and give a rigorous derivation.
Let us first describe the polling model studied in \cite{BKK}.

Consider a system of $m\ge 1$ infinite-buffer queues $Q_1,\ldots, Q_m$ and a
single server $S$.
Queues are indexed by $J=\{1,2,\ldots,m\}$.
The service times of customers in $Q_i$ are i.i.d. (independent, identically distributed) positive
random variables generically denoted by $B_i$, with means $b_i:=\ee B_i$.
Denote the Laplace-Stieltjes transform  (LST) of $B_i$ by $\tilde B_i(\cdot)$. The server moves among the queues
in a cyclic order. When $S$ moves from $Q_i$ to $Q_{i+1}$, it incurs a switchover period. The durations of successive switchover times are i.i.d. non-negative random variables, which we generically denote by $S_i$. Denote the LST of $S_i$ by $\tilde S_i(\cdot)$ and assume that $s_i:=\ee S_i<\infty$; let $s:=\sum_{i=1}^m s_i$.
Customers arrive at $Q_i$ according to a Poisson process with rate
$\lambda_i$; let $\lambda:=\sum_{i=1}^m\lambda_i$. We do not assume anything
about the service disciplines at $Q_i$.
Define $\rho_i:=\lambda_i b_i$ as
the traffic intensity at $Q_i$; let $\rho:=\sum_{i=1}^m\rho_i$. We assume that $\rho < 1$, which is a necessary condition for the system to be stable.
In what follows we shall write $\bz$ for
an $m$-dimensional vector in $\rr^m$, $\bz=(z_1,\ldots,z_m)$, and we assume that $|z_i| \leq 1$ for
every $i \in J$. We implicitly use the convention that any index summation is modulo $m$, for example $Q_{m+1}\equiv Q_1$.

Assume that all the usual independence assumptions hold between the service times,
the switch\-over times and the interarrival times. We assume that the ergodicity conditions are fulfilled
and we restrict ourselves to results for the stationary situation.

Now introduce the PGF of various joint queue-length distributions:
$V^b_i(\bz)$ and $V^c_i(\bz)$ denote the PGFs of the joint queue-length distribution at visit beginnings
and visit completions at $Q_i$, while $S^b_i(\bz)$ and $S^c_i(\bz)$ denote the PGFs of the joint
queue-length distribution at service beginnings and service completions at $Q_i$;
$L(\bz)$ denotes the PGF of the joint queue-length distribution at an arbitrary time in steady-state.
Theorem 1 of \cite{BKK} states that, with mean cycle time $\ee C = \frac{s}{1-\rho}$:
\begin{equation}
\label{eq:QL}
L(\bz)=
\frac{1}{\ee C}\sum_{i=1}^m\left(
\frac{V^b_i(\bz)-V^c_i(\bz)}{\Sigma(\bz)}
\frac{z_i\left(1-\tilde B_i(\Sigma(\bz))\right)}{z_i-\tilde B_i(\Sigma(\bz))}
+
\frac{V^c_i(\bz)-V^b_{i+1}(\bz)}{\Sigma(\bz)}
\right),
\end{equation}
with $\Sigma(\bz):=\sum_{j=1}^m \lambda_j(1-z_j)$.

Its proof in \cite{BKK} is based on the following relations:
\\
(i) a balance relation for polling systems, which is due to Eisenberg \cite{Eisenberg} and which was generalized
in \cite{BB}:
\begin{equation}
\label{eq:equilibrium}
\gamma_i V^b_i(\bz)+S^c_i(\bz)= S^b_i(\bz)+\gamma_i V^c_i(\bz), \quad i \in J.
\end{equation}
Here $\gamma_i:=1/\lambda_i\ee C$ represents the reciprocal of the mean number of customers served at $Q_i$
per visit, i.e., the long-term ratio of visit beginnings to service beginnings.
\\
(ii) an obvious relation between queue lengths at the beginning and end of a service time:
\begin{equation}
\label{eq:relation}
S^c_i(\bz)=S^b_i(\bz)\frac{\tilde B_i(\Sigma(\bz))}{z_i}, \quad i \in J.
\end{equation}
(iii) an obvious relation between queue lengths at the beginning and end of a switchover time:
\begin{equation}
\label{eq:rel}
V^b_{i+1}(\bz)=V^c_i(\bz)\tilde S_i\left(\Sigma(\bz)\right), \quad i \in J.
\end{equation}
(iv) a stochastic mean value theorem, expressing $L(\bz)$ as an average over the PGFs
of the joint queue-length distribution at an arbitrary moment during a visit to $Q_i$ ($X_i(\bz)$)
and during a switch\-over period between $Q_i$ and $Q_{i+1}$ ($Y_i(\bz)$):
\begin{equation}
\label{eq:smvt}
L(\bz)=\frac{1}{\ee C}\sum_{i=1}^m\left(\frac{b_i}{\gamma_i}
X_i(\bz)
+
s_iY_i(\bz)
\right),
\end{equation}
where, for $i \in J$,
\begin{align}
\label{eq:Xi}
 & X_i(\bz)=S^b_i(\bz) \tilde B_i^{\text{past}}(\Sigma(\bz)),\\
\label{eq:Yi}
 & Y_i(\bz)=V^c_i(\bz) \tilde S_i^{\text{past}}(\Sigma(\bz)),
\end{align}
where $\tilde B_i^{\text{past}}(\cdot)$ and $\tilde S_i^{\text{past}}(\cdot)$ are the LST's of the past (elapsed) parts of $B_i$ and $S_i$, respectively, that is, they are defined as
\begin{align*}
  \tilde B_i^{\text{past}}(\Sigma(\bz)) = \frac {1 - \tilde B_i(\Sigma(\bz))}{b_{i} \Sigma(\bz)}, \qquad \tilde S_i^{\text{past}}(\Sigma(\bz)) = \frac {1 - \tilde S_i(\Sigma(\bz))}{s_{i} \Sigma(\bz)}.
\end{align*}
Starting from (\ref{eq:smvt}), substituting (\ref{eq:Xi}) and (\ref{eq:Yi}),
and using (\ref{eq:equilibrium}) and (\ref{eq:relation}) to eliminate
all $S^c_i(\bz)$ and $S^b_i(\bz)$, yields (\ref{eq:QL}).
\begin{remark}{\rm
In \cite{BKK} also zero switchover times are allowed;
the same result (\ref{eq:QL}) is shown to hold.}
\end{remark}

In Theorem 1 of \cite{BKK} it was subsequently observed that one may simplify (\ref{eq:QL}) as follows, by using (\ref{eq:equilibrium})
and (\ref{eq:relation}):
\begin{equation}
\label{eq:simple}
L(\bz) = \frac{\sum_{i=1}^m \lambda_i (1-z_i) S^c_i(\bz)}{\sum_{i=1}^m \lambda_i(1-z_i)}.
\end{equation}

This formula is remarkably simple; please notice that it does not involve
the service time distributions, and that the service disciplines at the various queues also
do not play a role, which suggests that (\ref{eq:QL}) is based on very general principles.
This is the formula for which we would like to provide a short proof (see below).
In combination with (\ref{eq:equilibrium}) 
-- (\ref{eq:rel}), it also gives a short proof of (\ref{eq:QL}).
In other words, one can obtain an expression for the PGF of the joint steady-state queue-length distribution in a large class of polling systems
by just using the elementary balance equations (\ref{eq:equilibrium}) and (\ref{eq:simple2}) (see below), combined with the obvious
relations (\ref{eq:relation}) and (\ref{eq:rel}).
\\

\noindent
{\em Short proof of (\ref{eq:simple})}.
\\
First rewrite (\ref{eq:simple}) into
\begin{equation}
\label{eq:simple2}
\sum_{i=1}^m \lambda_i(1-z_i) L(\bz) = \sum_{i=1}^m \lambda_i (1-z_i) S^c_i(\bz).
\end{equation}
Secondly, observe that, because of the Poisson arrival processes,
$L(\bz)$ is also the PGF of the joint queue-length distribution just before an arrival at $Q_i$,
$i \in J$ by PASTA (Poisson Arrival See Time Averages, e.g., see \cite{BaccBrem2003,Miya1994}).
\\
Thirdly, invert the transform expressions on both sides of (\ref{eq:simple2}),
yielding for $\vc{x}\ge\vc{0}$ and $\vc{e}_i$ being the unit vector with $1$ in the $i$th coordinate and zero elsewhere:
\begin{align}
\sum_{i=1}^m \lambda_i \pi^{e}_i(\vc{x})-\sum_{i=1}^m \lambda_i \pi^{e}_i(\vc{x}-\vc{e}_i)
=\sum_{i=1}^m \lambda_i \pi_i^d(\vc{x})-\sum_{i=1}^m \lambda_i \pi_i^d(\vc{x}-\vc{e}_i),
\end{align}
where $\pi_i^d(\cdot)$ indicates that we consider the joint queue-length distribution
right {\em after} a departure from $Q_i$, and $\pi^{e}_i(\cdot)$ denotes that we view the system just \emph{before} an external arrival at $Q_i$.
Fourthly, we reshuffle the terms:
\begin{equation}
\sum_{i=1}^m \lambda_i \pi^{e}_i(\vc{x})+\sum_{i=1}^m \lambda_i \pi_i^d(\vc{x}-\vc{e}_i)
=\sum_{i=1}^m \lambda_i \pi_i^d(\vc{x})+\sum_{i=1}^m \lambda_i \pi_i^{e}(\vc{x}-\vc{e}_i).
\label{bal1}
\end{equation}
Finally, observe that the lefthand side of (\ref{bal1}) represents the rate
out of state $\vc{ x}$,
and the righthand side represents the rate into that state.
Indeed, the first term in the lefthand side corresponds to arrivals
which find $\vc{x}$ customers in the system.
The second term in the lefthand side  is slightly less obvious. It corresponds to departures that take place in state
$\vc{ x}$. Notice that the rate at which customers depart from $Q_i$ equals $\lambda_i$
(although the departure process will not be a Poisson process),
and that
$\pi_i^d(\vc{ x}-\vc{ e}_i)$ is the fraction of departures from $Q_i$ which
take the system  out of state $\vc{ x}$.
Similarly interpret the terms in the righthand side.
We conclude that (\ref{eq:simple}) amounts to a simple flow balance formula.
\begin{remark}{\rm
A similar flow balance argument was used in \cite{BT} to derive a queue-length expression
in an $M/G/1$ FCFS queue with multiple customer classes.}
\end{remark}
\begin{remark}{\rm
Observe that (\ref{eq:simple}) immediately gives the formula for the marginal distributions.
Indeed, for a vector $\bz_{m,i}=(1,\ldots,1,z_i,1,\ldots,1)$, $L(\bz_{m,i})=S^c_i(\bz_{m,i})$.
From the well-known `step' (level-crossing) argument it follows that $S^c_i(\bz_{m,i})$ is also the PGF of
the queue-length distribution in $Q_i$ at an {\em arrival} epoch
at $Q_i$. By PASTA it is also the PGF of the steady-state distribution of $Q_i$.

Next take $\bz_T=(z,\ldots,z)$. (\ref{eq:simple}) now states that
the PGF of the distribution of the total queue length (in terms of $z$) equals
$\sum_{i=1}^m \lambda_i S^c_i(\bz_T)/\sum_{j=1}^m \lambda_j$.
This formula may be interpreted as follows. By PASTA, $L(\bz_T)$ is also the PGF
of the distribution of the total queue length at an arrival epoch.
By a level-crossing argument, it follows that this equals the PGF of the distribution
of the total queue length just after a departure epoch.
The result now follows from the observation
that a fraction $\lambda_i / \sum_{j=1}^m \lambda_j$
of the departure epochs refers to a departure from $Q_i$.}
\end{remark}
\begin{remark}{\rm
Relation (\ref{eq:simple}) may be viewed as an $m$-dimensional version of the above-mentioned one-dimensional
`step' (level-crossing) relation that holds for queues with single arrivals and single departures.}
\end{remark}

\section{Formal derivations under a general framework}
\label{sect:formal}

In this section, we aim to derive
distributional relationships at arrival and departure instants for various queues and their network models
in a unified way, under general settings. Roughly speaking, these settings allow simultaneous external arrivals, simultaneous departures and routing at different stations; however, we do not allow an external arrival to coincide with a departure. We use their time evolutions in sample paths for deriving the relationships rather than using flow balance.

We describe a queueing network system under a fairly general framework. We consider an open queueing network system with $m$ queues, where queues uniquely belong to service facilities, which are called stations.
Queues in the same station may be distinguished by customer classes. Each station may have multiple servers, which may change in time. External arrivals at queues are general as long as they satisfy certain stationarity conditions. Customers completing service may be routed among queues depending on the state of the whole system. Thus, this model is quite general and very flexible.

To describe this model, we introduce a stochastic process.
Queues are still indexed by $J=\{1,2,\ldots,m\}$.
Let
\begin{align*}
  \vc{X}(t)=(X_1(t),\ldots,X_m(t)),
\end{align*}
where $X_{i}(t)$ represents the length of queue $i$ at time $t$, which includes customers in service. Here, each queue belongs to a single station. There is a mapping from queues to stations, which will be given when needed.

In addition to $\vc{X}(t)$, the following counting processes count the number of specified events until time $t\ge0$ for $i\in J$,
\begin{itemize}
\item $N^{e}_{i}(t)$ - external arrivals at queue $i$,
\item $N^{d}_{i}(t)$ - departures from queue $i$,
\item $N^{r}_{i}(t)$ - internal arrivals at queue $i$ (transition from some queue).
\end{itemize}

With $\vc{N}^u(t)=(N_1^u(t),\ldots,N_m^u(t))$ for $u=e,d,r$, we consider the process
\begin{align*}
  \vc{Z}(t) \equiv (\vc{X}(t),\vc{N}^{e}(t),\vc{N}^{d}(t),\vc{N}^{r}(t))\ .
\end{align*}
All processes are assumed right-continuous with left limits. Let $\Delta \vc{X}(t) = \vc{X}(t) - \vc{X}(t-)$. $\Delta \vc{N}^{u}(t)$ is similarly defined and is in $\dd{Z}_+^m$ for $u = e,d,r$,
where $\dd{Z}_{+}$ is the set of nonnegative integers.

For $u=e,d,r$, denote
\begin{align*}
  |\vc{N}^{u}|(t) = \sum_{i \in J} N^{u}_{i}(t)
\end{align*}
and assume that
\begin{itemize}
\item [i)] $\vc{X}(0), \vc{N}^{e}(t),\vc{N}^{d}(t),\vc{N}^{r}(t)$ are all finite (in $\dd{Z}_+^m$) for each $t \ge 0$.
\item [ii)] $\Delta |\vc{N}^{e}|(t) \Delta |\vc{N}^{d}|(t) = 0$ for each $t \geq 0$. That is, external arrivals and service completions can not occur simultaneously.
\end{itemize}
We also need to define the intermediate state
\begin{align}
\label{eq:dynamics 1}
  \vc{X}^{d}(t) = \vc{X}(t) - \Delta \vc{N}^{r}(t)\in\dd{Z}_+^m\ .
\end{align}
It differs from $\vc{X}(t)$ only at departure epochs and it describes the state ``after'' a departure and ``before'' an internal arrival at a different queue.

Clearly, the following dynamics hold.
\begin{align}
\label{eq:dynamics 2}
 \vc{X}(t) = \vc{X}(0) + \vc{N}^{e}(t) - \vc{N}^{d}(t) + \vc{N}^{r}(t)\in\dd{Z}_+^m\ .
\end{align}

Because of i), $\vc{X}(t)$ and $\vc{X}^{d}(t)$ are also finite. It may be natural to assume that $|\vc{N}^{r}|(t) \le |\vc{N}^{d}|(t)$ for $t \ge 0$, but we do not require it in this section.

Thus, $\vc{X}(t)$ is the state of the system at time $t$ of an input-output system driven by counting processes $\vc{N}^{e},\vc{N}^{d},\vc{N}^{r}$. The dynamics of \eq{dynamics 1} and \eq{dynamics 2} indicates that we adopt the {\em departure first} framework. We have used queueing terminologies, but our results are valid as long as the above mathematical assumptions and \eq{dynamics 2} are satisfied.

In general, $|\vc{N}^{e}|(t), |\vc{N}^{d}|(t)$ and $N^{e}_{i}(t)$ and $N^{d}_{i}(t)$ may have jumps greater than one, which is not convenient to describe the time evolution of $\vc{Z}(t)$. Thus, for $u=e,d$, we introduce
\begin{equation*}
|\tilde{\vc{N}}^u|(t) = \sum_{0<s\le t} 1(\Delta|\vc{N}^u|(s)\ge 1),
\qquad {\tilde N}^{u}_{i}(t) = \sum_{0<s\le t} 1(\Delta N^u_i(s)\ge 1),
\end{equation*}
then $\Delta |\tilde{\vc{N}}^{u}|(t) \le 1$ and $\Delta \tilde N^{u}_{i}(t) \le 1$, that is, $|\tilde{\vc{N}}^{u}|$ and $\tilde N^{u}_{i}$ are simple point processes for $u=e,d$. Set $t_0^e=t_0^d=t_{i,0}^e=t_{i,0}^d=0$ for $i\in J$, and for $n\ge 1$ and $i\in J$ let $t^e_n,t^d_n,t^e_{i,n},t^d_{i,n}$ be the $n^{\text{th}}$ jump epoch of $|\tilde{\vc{N}}^e|,|\tilde{\vc{N}}^d|,\tilde{N}^e_i,\tilde{N}^d_i$, respectively (of course, if the corresponding process is not terminating and such an epoch exists).

Another basic assumption on the counting processes is
\begin{itemize}
\item [iii)] There exist finite and positive numbers $\lambda^{u}$, $u=e,d$ such that
\begin{align}
\label{eq:rate 1}
  \lambda^{u} = \lim_{t \to \infty} \frac 1t |\tilde{\vc{N}}^{u}|(t)\ ,
\end{align}
a.s. (almost surely) w.r.t. the underlying probability measure $\dd{P}$.
\end{itemize}

We further assume the following ergodic type conditions.
\begin{itemize}
\item [iv)] There exist probability distributions $\pi^{e}$ and $\pi^{d}$ such that
\begin{align}
\label{eq:stationary 1}
 & \lim_{n \to \infty} \frac 1n \sum_{\ell=1}^{n} 1(\vc{X}(t^{e}_{\ell}-) = \vc{x}, \Delta \vc{N}^{e}(t^{e}_{\ell}) = \vc{y}) = \pi^{e}(\vc{x},\vc{y}), \quad a.s., \quad \vc{x}, \vc{y} \in \dd{Z}_+^m,\\
\label{eq:stationary 2}
 & \lim_{n \to \infty} \frac 1n \sum_{\ell=1}^{n} 1(\vc{X}^{d}(t^{d}_{\ell}) = \vc{x}, \Delta \vc{N}^{d}(t^{d}_{\ell}) = \vc{y}, \Delta \vc{N}^{r}(t^{d}_{\ell}) = \vc{z}) = \pi^{d}(\vc{x},\vc{y},\vc{z}), \quad a.s.,\nonumber\\
 & \hspace{45ex}\qquad \vc{x}, \vc{y}, \vc{z} \in \dd{Z}_+^m .
\end{align}
\end{itemize}

From the definitions in iv), $\pi^{e}$ and $\pi^{d}$ are considered as the embedded stationary distributions just before arrival epochs and just after departure epochs but before internal arrivals, respectively. They correspond to Palm distributions concerning their counting processes in the time stationary framework (e.g., see \cite{BaccBrem2003}).

Since the process $\vc{X}(t)$ is vector valued, it is not so convenient for manipulations. So, we introduce a test function
$f:\dd{Z}_{+}^{m}\to\dd{R}$.
Under the setting i)--iv), we will derive distributional relationships among characteristics at different embedded instants using the test function $f$. For this, we need the following lemma.
\begin{lemma}
\label{lem:stationary 1}
If \eq{stationary 1} holds, then, for any bounded function $g:\dd{Z}_+^{2m}\to\dd{R}$, we have
\begin{align}
\label{eq:stationary f1}
 & \lim_{n \to \infty} \frac 1n \sum_{\ell=1}^{n} g(\vc{X}(t^{e}_{\ell}-),\Delta \vc{N}^{e}(t^{e}_{\ell})) = \sum_{\vc{x},\vc{y} \in \dd{Z}_+^m} g(\vc{x},\vc{y}) \pi^{e}(\vc{x},\vc{y}), ~~~ {\rm a.s.}
\end{align}
Similarly, if \eq{stationary 2} holds, then, for any bounded function $h:\dd{Z}_+^{3m}\to\dd{R}$, we have
\begin{align}
\label{eq:stationary g1}
 & \lim_{n \to \infty} \frac 1n \sum_{\ell=1}^{n} h(\vc{X}^{d}(t^{d}_{\ell}),\Delta \vc{N}^{d}(t^{d}_{\ell}),\Delta \vc{N}^{r}(t^{d}_{\ell})) = \sum_{\vc{x},\vc{y},\vc{z} \in \dd{Z}_+^m} h(\vc{x},\vc{y},\vc{z}) \pi^{d}(\vc{x},\vc{y},\vc{z}), ~~~ {\rm a.s.}
\end{align}
\end{lemma}

This lemma may look obvious, but its proof is not immediate because we need to verify the exchange of limits. We prove it in \app{stationary 1}.

We are now ready to prove distributional relationships. First, we denote the expectations under $\pi^{e}$ and $\pi^{d}$ by $\dd{E}^{e}$ and $\dd{E}^{d}$, respectively. That is,
\begin{align}
\label{eq:Palm e}
 & \dd{E}^{e}g(\vc{X}, \vc{Y}) = \sum_{\vc{x},\vc{y} \in \dd{Z}_+^m} g(\vc{x},\vc{y}) \pi^{e}(\vc{x},\vc{y}), \\
\label{eq:Palm d}
 & \dd{E}^{d}h(\vc{X},\vc{Y},\vc{Z}) = \sum_{\vc{x},\vc{y},\vc{z} \in \dd{Z}_+^m} h(\vc{x},\vc{y},\vc{z}) \pi^{d}(\vc{x},\vc{y},\vc{z}).
\end{align}
Note that $\vc{Y}$ in $\dd{E}^{e}$ represents sizes of externally arriving batches, while $\vc{Y}$ in $\dd{E}^{d}$ represents sizes of departing batches.
\begin{theorem}
\label{thr:stationary relation 1}
Under the setting i)--iv), for any bounded function $f:\dd{Z}_+^m\to\dd{R}$, we have
\begin{equation}
\label{eq:stationary relation 1}
  \lambda^{e} \dd{E}^{e}\big[f(\vc{X} + \vc{Y}) - f(\vc{X})\big]  + \lambda^{d} \dd{E}^{d}\big[f(\vc{X} + \vc{Z}) - f(\vc{X})\big]
   = \lambda^{d} \dd{E}^{d} \big[f(\vc{X} + \vc{Y}) - f(\vc{X}) \big].
\end{equation}
\end{theorem}
\begin{proof}
Since $f(\vc{X}(t))$ changes in time only at the counting instants $t^{e}_{n}$ or $t^{d}_{n}$, we have
(with $\Delta$ being defined as earlier in this section)
\begin{align}
\label{eq:test 1}
  f(\vc{X}(t)) - f(\vc{X}(0)) & = \sum_{\ell = 1}^{|\tilde{\vc{N}}^{e}|(t)} \Delta f(\vc{X}(t^{e}_{\ell})) + \sum_{\ell = 1}^{|\tilde{\vc{N}}^{d}|(t)} \Delta f(\vc{X}(t^{d}_{\ell})).
\end{align}
Recalling \eq{dynamics 1}, we have $\vc{X}(t^{d}_{\ell}) = \vc{X}^{d}(t^{d}_{\ell}) + \Delta \vc{N}^{r}(t^{d}_{\ell})$, and this and \eq{dynamics 2} yield
\begin{align*}
  \vc{X}(t^{d}_{\ell}-) & = \vc{X}(t^{d}_{\ell}) + \Delta \vc{N}^{d}(t^{d}_{\ell}) - \Delta \vc{N}^{r}(t^{d}_{\ell}) = \vc{X}^{d}(t^{d}_{\ell}) + \Delta \vc{N}^{d}(t^{d}_{\ell}).
\end{align*}
Substituting these $\vc{X}(t^{d}_{\ell})$ and $\vc{X}(t^{d}_{\ell}-)$ into \eq{test 1}, we have
\begin{align}
\label{eq:test 2}
  \sum_{\ell = 1}^{|\tilde{\vc{N}}^{d}|(t)} \Delta f(\vc{X}(t^{d}_{\ell})) & = \sum_{\ell = 1}^{|\tilde{\vc{N}}^{d}|(t)} (f(\vc{X}(t^{d}_{\ell})) - f(\vc{X}^{d}(t^{d}_{\ell}))) + \sum_{\ell = 1}^{|\tilde{\vc{N}}^{d}|(t)} (f(\vc{X}^{d}(t^{d}_{\ell})) - f(\vc{X}(t^{d}_{\ell}-))) \nonumber\\
  & = \sum_{\ell = 1}^{|\tilde{\vc{N}}^{d}|(t)} (f(\vc{X}^{d}(t^{d}_{\ell})+\Delta \vc{N}^{r}(t^{d}_{\ell})) - f(\vc{X}^{d}(t^{d}_{\ell}))) \nonumber\\
  & \qquad + \sum_{\ell = 1}^{|\tilde{\vc{N}}^{d}|(t)} (f(\vc{X}^{d}(t^{d}_{\ell})) - f(\vc{X}^{d}(t^{d}_{\ell}) + \Delta \vc{N}^{d}(t^{d}_{\ell}))) .
\end{align}
It follows from \eq{test 1} and \eq{test 2} that
\begin{align}
\label{eq:test 3}
 \sum_{\ell = 1}^{|\tilde{\vc{N}}^{e}|(t)} & (f(\vc{X}(t^{e}_{\ell}-) + \Delta \vc{N}^{e}(t^{e}_{\ell})) - f(\vc{X}(t^{e}_{\ell}-))) + \sum_{\ell = 1}^{|\tilde{\vc{N}}^{d}|(t)} (f(\vc{X}^{d}(t^{d}_{\ell})+\Delta \vc{N}^{r}(t^{d}_{\ell})) - f(\vc{X}^{d}(t^{d}_{\ell}))) \nonumber\\
 & = \sum_{\ell = 1}^{|\tilde{\vc{N}}^{d}|(t)} (f(\vc{X}^{d}(t^{d}_{\ell}) + \Delta \vc{N}^{d}(t^{d}_{\ell})) - f(\vc{X}^{d}(t^{d}_{\ell}))) + f(\vc{X}(t)) - f(\vc{X}(0)).
\end{align}
Dividing both sides of this equation by $t$ and letting $t \to \infty$ yields \eq{stationary relation 1} by \eq{rate 1}--\eq{stationary 2} and \lem{stationary 1} because $f$ is bounded.
\end{proof}

The assumptions of \thr{stationary relation 1} exclude arrivals and departures to occur simultaneously, but allow them to occur separately as multiple simultaneous external arrivals or multiple simultaneous departures and routing. The model as well as the distributional relationship may be too general for queueing networks. To make them more specific, we make the following assumption.
\begin{itemize}
\item [v)] There exist finite and nonnegative numbers $\lambda^{d}_{A}$ for nonempty $A \subset J$, that is, $A \in 2^{J} \setminus \{\emptyset\}$, such that
\begin{align}
\label{eq:rate 2}
  \lambda^{d}_{A} = \lim_{t \to \infty} \frac 1t \tilde{N}^{d}_{A}(t), \qquad a.s.,
\end{align}
where, with notation $S_{A} \equiv \{\vc{x} \in \dd{Z}_{+}^{m}; x_{i} > 0, i \in A, x_{j} = 0, j \in J \setminus A\}$,
\begin{align}
\label{eq:NA 1}
  \tilde{N}^{d}_{A}(t) = \sum_{0<s\le t} 1(\Delta \vc{N}^d(s) \in S_{A} ).
\end{align}
\end{itemize}

Note that $\tilde{N}^{d}_{A}$ counts instants when departures occur simultaneously from queues $i \in A$ but there is no departure from queue $j \in J \setminus A$, while $\Delta \tilde{N}^{d}_{A}(t) \Delta \tilde{N}^{d}_{B}(t) = 0$ if $A \ne B$. Thus, the setting i)--v) still allows batch arrivals and batch departures and simultaneous transfer of customers in a departing batch.

We will use the following notation. For each $A \in 2^{J} \setminus \{\emptyset\}$, let, for $\vc{x}, \vc{y}, \vc{z} \in \dd{Z}_{+}^{m}$,
\begin{align*}
\pi^{d}_{A}(\vc{x},\vc{y},\vc{z}) = \left\{
\begin{array}{ll}
  \frac {\lambda^{d}}{\lambda^{d}_{A}} \pi^{d}(\vc{x}, \vc{y},\vc{z}), \quad & \lambda_{A}^d > 0,  \\
  0, & \lambda_{A}^d = 0.
\end{array}
\right.
\end{align*}
Since $\tilde{N}^{d}_{A}$ exclusively counts the increasing epochs of $|\tilde{\vc{N}}^{d}|$ for different $A$'s, we have
\begin{align}
\label{eq:exclusive 1}
  |\tilde{\vc{N}}^{d}|(t) = \sum_{A \in 2^{J} \setminus \{\emptyset\}} \tilde{N}^{d}_{A}(t),
\end{align}
which implies that $\lambda^d = \sum_{A \in 2^J \setminus \{\emptyset\}} \lambda^d_{A}$, and, for $A \in \{B \in 2^{J}| \lambda^d_{B} > 0\}$, $\pi^{d}_{A}$ is a probability distribution on $\dd{Z}_{+}^{3m}$, which can be restricted to $\dd{Z}_{+}^{m} \times S_{A} \times \dd{Z}_{+}^{m}$.

Let $t^{d}_{A,n}$ be the $n^{\text{th}}$ jump epoch of $\tilde{N}^{d}_{A}$. Just as \lem{stationary 1} does, the following lemma plays a key role; it is proved in \app{single 1}.

\begin{lemma}
\label{lem:single 1}
Under the setting i)--v), there exist probability distributions $\pi^{d}_{A}$ such that, for any bounded function $h:\dd{Z}_+^{3m}\to\dd{R}$,
with
$A \in 2^J \setminus \{\emptyset\}$,
\begin{align}
\label{eq:stationary local 1}
 & \lim_{n \to \infty} \frac 1n \sum_{\ell=1}^{n} h(\vc{X}^{d}(t^{d}_{A,\ell}), \Delta \vc{N}^{d}(t^{d}_{A,\ell}), \Delta \vc{N}^{r}(t^{d}_{A,\ell})) = \sum_{\vc{x},\vc{y},\vc{z} \in \dd{Z}_{+}^{m}} h(\vc{x},\vc{y},\vc{z}) \pi^{d}_{A}(\vc{x},\vc{y},\vc{z}), ~ {\rm a.s.}
\end{align}
\end{lemma}

By \eq{exclusive 1},
\thr{stationary relation 1} and \lem{single 1} yield the following corollary. As with $\dd{E}^{e}$ and $\dd{E}^{d}$, $\dd{E}^d_{A}$ stands for the expectation under $\pi^{d}_{A}$.

\begin{cor}
\label{cor:stationary relation 1}
Under the setting i)--v), for any bounded function $f:\dd{Z}_+^m\to\dd{R}$,
\begin{align}
\label{eq:stationary relation 2}
  \lambda^{e} \dd{E}^{e}\left[f(\vc{X} + \vc{Y}) - f(\vc{X})\right] & + \sum_{A \in 2^{J}\setminus \{\emptyset\}} \lambda^{d}_{A} \dd{E}^d_{A}\left[f(\vc{X} + \vc{Z}) - f(\vc{X})\right]\nonumber\\
  & = \sum_{A \in 2^{J} \setminus \{\emptyset\}} \lambda^{d}_{A} \dd{E}^d_{A} \left[f(\vc{X} + \vc{Y}) - f(\vc{X}) \right].
\end{align}
\end{cor}
\begin{remark}{\rm
\label{rem:stationary relation 1}
If $\Delta \tilde{N}^{d}_{i}(t^{d}_{j,n}) = 0$ for all $i \ne j$, then $\lambda^{d}_{A} > 0$ only if $A$ is a singleton. In this case, the summations over $A$ in \eq{stationary relation 2} can be reduced to those over $i \in J$, replacing $A$ by $i$.}
\end{remark}

Until now, our distributional relationship may still be too general because no assumption is made on how the counting processes are generated from $\vc{X}(t)$ and other information. To describe this, a filtration is convenient. Let $\sr{F}_{t}$ be the $\sigma$-field generated by all events up to time $t$, and let $\sr{F}_{t-} = \sigma(\cup_{u < t} \sr{F}_{u})$,
that is, $\sr{F}_{t-}$ is a $\sigma$-field generated by all events before time $t$.
For a stopping time $\tau$, let $\sr{F}_{\tau-} = \sigma (\sr{F}_{0}, \{A \cap \{t < \tau\} \in \sr{F}_{t}\})$, where $\sigma(\sr{A})$ is the $\sigma$-field generated by a family of events $\sr{A}$. Using the filtration, the following assumptions are typically used under the setting i)--v).
\begin{itemize}
\item [(a1)] $t^{e}_{n}, t^{d}_{i,n}$ are stopping times with respect to $\{\sr{F}_{t}; t \ge 0\}$. This can always be realized by choosing a sufficiently large $\sr{F}_{t}$.
\item [(a2)] $\Delta \vc{N}^{e}(t^{e}_{n})$ is independent of $\sr{F}_{t^{e}_{n}-}$. That is, the sizes of batch arrivals are independent of the state of the system just before their arrival epochs.
\item [(a3)] $\Delta |\vc{N}^{d}|(t^{d}_{n}) = 1$. That is, departures singly occur from one queue at a time.
\item [(a4)] $\Delta N^{r}_{j}(t^{d}_{i,n}) \le 1$ for $j \in J$, and $\Delta \vc{N}^{r}(t^{d}_{i,n})$ is in the $\sigma$-field generated by $\sr{F}_{t^{d}_{i,n}-}$ and $\Delta \vc{N}^{d}(t^{d}_{i,n})$.
\end{itemize}

By (a3), $\tilde{N}^{d}_{A}(t) \equiv 0$ if $A$ is not a singleton. Thus, we write $\tilde{N}^{d}_{A}(t)$ as $\tilde{N}^{d}_{i}(t)$ for $A = \{i\}$. Similarly, $\pi^{d}_{A}$ is written as $\pi^{d}_{i}$ for $A = \{i\}$. Under the setting i)--v) and the assumptions (a1)--(a4),  $\Delta N^{r}_{j}(t^{d}_{i,\ell}) \le 1$, and therefore \lem{single 1} yields
\begin{align*}
  \lim_{n \to \infty} \frac 1n \sum_{\ell=1}^{n} 1(\vc{X}^{d}(t^{d}_{i,\ell}) = \vc{x}, \Delta N^{d}_{i}(t^{d}_{i,\ell}) = 1, \Delta N^{r}_{j}(t^{d}_{i,\ell}) = 1) = \pi^{d}_{i}(\vc{x},\vc{e}_{i}, \vc{e}_{j}),
\end{align*}
which is denoted by $\pi^{d}_{ij}(\vc{x})$. We here recall that $\vc{e}_{i} \in \dd{Z}_{+}^{m}$ is the unit vector whose $i$-th entry is one and the other entries are zero. Thus, applying Corollary \cort{stationary relation 1} for $f(\vc{x}) = \vc{z}^{\vc{x}}$, where we recall that $\vc{z}^{\vc{x}} = \prod_{i \in J} z_{i}^{x_{i}}$, we have the following relationship.
\begin{cor}
\label{cor:stationary relation 2}
Under the settings i)--v) and assumptions (a1)--(a4), for $\vc{z} = (z_{1}, \ldots, z_{m})$ satisfying $|z_{i}| \le 1$ for $i \in J$,
\begin{align}
\label{eq:stationary relation 3}
  \lambda^{e} & \big(1 - \dd{E}^e[\vc{z}^{\vc{Y}}] \big)\varphi^{e}(\vc{z}) + \sum_{j \in J} (1-z_{j}) \sum_{i \in J} \lambda^{d}_{i} \varphi^{d}_{ij}(\vc{z}) = \sum_{i \in J} \lambda^{d}_{i} (1 - z_{i}) \varphi^{d}_{i}(\vc{z}),
\end{align}
where
\begin{align*}
 \varphi^{e}(\vc{z}) = \dd{E}^{e}[\vc{z}^{\vc{X}}],\qquad \varphi^{d}_{i}(\vc{z}) = \dd{E}^d_i[\vc{z}^{\vc{X}}], \qquad \varphi^{d}_{i,j}(\vc{z}) = \sum_{\vc{x} \in \dd{Z}_+^m} \vc{z}^{\vc{x}} \pi^{d}_{ij}(\vc{x}), \qquad i,j \in J.
\end{align*}
\end{cor}
\begin{remark}{\rm
\label{rem:stationary relation 2}
Under the assumptions of this corollary, the routing of departing customers may depend on all queue lengths in the network.}
\end{remark}

Corollary \cort{stationary relation 2} is specialized to Corollary \cort{stationary relation 3} if external arrivals to queues occur one at a time. Namely,
\begin{itemize}
\item [vi)] No simultaneous arrivals occur, and there exist finite numbers (some, but not all, possibly zero) $\lambda^{e}_{k}$ for $k \in J$ such that
\begin{align}
\label{eq:rate 3}
  \lambda^{e}_{k} = \lim_{t \to \infty} \frac 1t \tilde{N}^{e}_{k}(t), \quad a.s., \qquad k \in J.
\end{align}
\end{itemize}

\begin{cor}
\label{cor:stationary relation 3}
Under the assumptions of Corollary \cort{stationary relation 2}, assume that vi) also holds. Define the $\pi^{e}_{k}$ as
\begin{align*}
  \pi^{e}_{k}(\vc{x},y_{k}) = \left\{
\begin{array}{ll}
  \frac {\lambda^{e}}{\lambda^{e}_{k}} \pi^{e}(\vc{x},y_{k}), \quad & \lambda_{k}^{e} > 0,\\
 0, & \lambda_{k}^{e} = 0,
\end{array}
\right.
\end{align*}
then, for $k \in \{i \in J| \lambda_{i}^{e} > 0\}$, the $\pi^{e}_{k}$
is a probability distribution on $\dd{Z}_+^{m+1}$, and \eq{stationary relation 3} becomes
\begin{align}
\label{eq:stationary relation 4}
  \sum_{k \in J} \lambda^{e}_{k} \big(1 - \dd{E}^e\big[z_{k}^{Y_{k}}\big] \big) \varphi^{e}_{k}(\vc{z}) + \sum_{j \in J} (1-z_{j}) \sum_{i \in J} \lambda^{d}_{i} \varphi^{d}_{ij}(\vc{z}) = \sum_{i \in J} \lambda^{d}_{i} (1 - z_{i}) \varphi^{d}_{i}(\vc{z}),
\end{align}
where $\varphi^{e}_{k}$ is the generating function of $\vc{X}$ under the conditional distribution $\pi^{e}_{k}$.
\end{cor}

Corollary~\ref{cor:stationary relation 3} immediately implies the following corollary.

\begin{cor}
\label{cor:stationary relation 4}
Under the assumptions of Corollary \cort{stationary relation 3}, if the event $\{\Delta N^{r}_{j}(t^{d}_{i,\ell}) = 1\}$ is independent of $\sr{F}_{t^{d}_{i,\ell}-}$, then there exist $p_{ij} \ge 0$ such that $\pi^{d}_{i}(\vc{x},1, \vc{e}_{j}) = \pi^{d}_{i}(\vc{x},1) p_{ij}$, and \eq{stationary relation 4} becomes
\begin{align}
\label{eq:stationary relation 5}
  \sum_{k \in J_{e}} \lambda^{e}_{k} \big(1 - \dd{E}^e\big[z_{k}^{Y_{k}}\big] \big)\varphi^{e}_{k}(\vc{z}) + \sum_{i \in J} \lambda^{d}_{i} \varphi^{d}_{i}(\vc{z}) \sum_{j \in J} p_{ij} (1-z_{j}) = \sum_{i \in J} \lambda^{d}_{i} (1 - z_{i}) \varphi^{d}_{i}(\vc{z}).
\end{align}
\end{cor}

\begin{remark}{\rm
In Section~\ref{sec:3} we shall present several applications of the above theorem and corollaries.
In particular, the polling result (\ref{eq:simple}) of Section~\ref{sec:2} is there shown to be a special case of Corollary~\ref{cor:stationary relation 1}.
\\
Notice that the setup of this section includes the finite buffer case. This is done by having no arrivals to a queue during times in which it is saturated. This type of dependence is allowed by our setup. Some results for the single server queue with finite capacity are contained in \cite{HebuterneRosenberg}.}
\end{remark}

\section{Distributional relationship up to a given time}
\label{sect:non stationary}

The purpose of this section is to derive a non-stationary version of \thr{stationary relation 1}, a distributional relationship
{\em up to a given time}.
We adopt the settings i)--iv)
of \sectn{formal}, and consider the process $\vc{Z}(t)$ introduced in the beginning of that section. We first define the expected relative frequencies for bounded test functions $g, h$ from $\dd{Z}_+^{2m}, \dd{Z}_+^{3m}$ to $\dd{R}$ up to time $t$ as
\begin{align*}
 & R^{e}_{t}g = \frac 1{|\tilde{\vc{N}}^{e}|(t)} \sum_{n=1}^{|\tilde{\vc{N}}^{e}|(t)} g(\vc{X}(t^{e}_{n}-),\Delta \vc{N}^{e}(t^{e}_{n})) 1(|\tilde{\vc{N}}^{e}|(t) > 0),\\
 & R^{d}_{t} h = \frac 1{|\tilde{\vc{N}}^{d}|(t)} \sum_{n=1}^{|\tilde{\vc{N}}^{e}|(t)} h(\vc{X}^{d}(t^{d}_{n}),\Delta \vc{N}^{d}(t^{d}_{n}),\Delta \vc{N}^{r}(t^{d}_{n})) 1(|\tilde{\vc{N}}^{d}|(t) > 0).
\end{align*}
For each bounded function $f:\dd{Z}_+^{m}\to\dd{R}$, we define the following test functions.
\begin{align*}
   & g^{e}(\vc{x},\vc{y}) = f(\vc{x}), \qquad g^{e}_+(\vc{x}, \vc{y}) = f(\vc{x} + \vc{y}),\\
   & h^{d}(\vc{x},\vc{y},\vc{z}) = f(\vc{x}), \qquad h^{d}_-(\vc{x},\vc{y},\vc{z}) = f(\vc{x} + \vc{y}), \qquad h^{d}_{+}(\vc{x},\vc{y},\vc{z}) = f(\vc{x} + \vc{z}).
\end{align*}
 Let
\begin{align*}
  \lambda^{e}(t) = \frac 1t |\tilde{\vc{N}}^{e}|(t), \qquad \lambda^{d}(t) = \frac 1t |\tilde{\vc{N}}^{d}|(t).
\end{align*}

Then, \eq{test 3} yields the following lemma.
\begin{lemma}
\label{lem:nonstationary relation 1}
Under the setting i)--iv), for any bounded function $f:\dd{Z}_+^{m}\to\dd{R}$, we have, for any $t > 0$,
\begin{align}
\label{eq:nonstationary relation 1}
  \lambda^{e}(t) \big(R^{e}_{t} g^{e}_+ & - R^{e}_{t} g^{e}\big) + \lambda^{d}(t) \big(R^{d}_{t} h^{d}_+ - R^{d}_{t} h^d\big)\nonumber\\
 &  - \lambda^{d}(t) \big(R^{d}_{t} h^{d}_- - R^{d}_{t} h^{d}\big) = \frac 1t \big(f(\vc{X}(t)) - f(\vc{X}(0))\big).
\end{align}
\end{lemma}

We may interpret \lem{nonstationary relation 1} as a transient version of \thr{stationary relation 1}. It is notable that \eq{nonstationary relation 1} holds without any stability condition, and its right-hand side vanishes as $t \to \infty$ at most in linear order of $t^{-1}$ because $f$ is bounded. If there exists a unique probability measure such that $(\vc{X}(t), \Delta \vc{N}^{e}(t), \Delta \vc{N}^{d}(t), \Delta \vc{N}^{r}(t))$ is stationary, then $R^{e}_{t} g, R^{d}_{t}h$ converge to the corresponding expectations under the Palm distributions involving $|\tilde{\vc{N}}^{e}|, |\tilde{\vc{N}}^{d}|$, respectively. Thus, we have
\begin{align*}
 & \lim_{t \to \infty} R^{e}_{t} g^{e} = \dd{E}^{e}f(\vc{X}), \qquad \lim_{t \to \infty} R^{e}_{t} g^{e}_+ = \dd{E}^{e}f(\vc{X} + \vc{Y}),\\
 & \lim_{t \to \infty} R^{d}_{t} h^{d}(\vc{x}) = \dd{E}^{d}f(\vc{X}), \qquad \lim_{t \to \infty} R^{d}_{t} h^{d}_- = \dd{E}^{e}f(\vc{X} + \vc{Y}),\\
 & \lim_{t \to \infty} R^{d}_{t} h^{d}_+ = \dd{E}^{d}f(\vc{X} + \vc{Z}),
\end{align*}
and we recover \eq{stationary relation 1} from \eq{nonstationary relation 1}. Corollary \cort{stationary relation 1}, \eq{stationary relation 3} and \eq{stationary relation 4} are similarly obtained. We omit the routine details.

\section{Some special cases and applications}
\label{sec:3}
In this section we consider several applications of the theorem and corollaries of Section~\ref{sect:formal}.
We first note that, if nonzero $N^{e}_{k}$ for $k \in J$ are independent compound Poisson processes,
then by PASTA the embedded stationary distributions $\pi^{e}$ and $\pi^{e}_{k}$ are identical with the time stationary distributions.
\\

\noindent
{\em Case $1$: An $m$-class queue with batch arrivals}
\\
We consider an $m$-class single-node service facility, with $m\geq 1$.
We allow multiple servers.
Customers arrive according to a Poisson process, possibly in batches.
Customers of class $i$ require service at the service facility according to service time distribution $B_i(\cdot)$,
$i \in J$.
These distributions are assumed to be continuous, but not otherwise specified.
No customers are lost; there is an infinite waiting room.
After completion of their service, customers immediately leave.
We assume that the steady-state joint queue-length distribution
(numbers of customers of all classes in the system) exists. Its PGF is denoted by
$L(\bz)$.
We also again (as in Section~\ref{sec:2}) denote the PGF of the steady-state
joint queue-length distributions immediately after
departure epochs of a class $i$-customer by $S^c_i(\bz)$, $i \in J$.
We do not specify according to which service discipline the customers are served;
polling with FCFS within each class is just one of many options.

\begin{theorem}
\label{thm:batcharrival}
Consider the above-described $m$-class single-node service facility.
Assume that
customers arrive according to a batch Poisson process with rate $\lambda$
and that customers are served individually, in some non-specified order.
Let an arbitrary batch arrival have size $\bG = (G_1,\dots,G_m)$ with
PGF $\ee [\bz^{\bG}] = \ee [z_1^{G_1} \dots z_m^{G_m}]$.
Then the following relation holds between the PGF $L(\bz)$
and the PGFs $S^c_i(\bz)$,
$i \in J$:
\begin{equation}
(1 - \ee [\bz^{\bG}]) L(\bz) = \sum_{i=1}^m (1-z_i) \ee G_i S^c_i(\bz).
\label{eq:simple3}
\end{equation}
\end{theorem}

\begin{proof}
After using PASTA,
Theorem \ref{thm:batcharrival} is a special case of \eq{stationary relation 3} of Corollary \cort{stationary relation 2} in which there is no routing.
\end{proof}
\begin{remark}{\rm
Special cases of the above theorem are obtained by assuming that batches always contain
only customers of one type. For the special case that batches have just one customer of class $i$ with probability
$\frac{\lambda_i}{\lambda}$, $i \in J$, (\ref{eq:simple3}) reduces to (\ref{eq:simple}) that was obtained for the polling system
that provided the initial motivation for the present study
(but (\ref{eq:simple}) obviously holds for a much more general class of service disciplines).}
\end{remark}
\noindent
{\em Case $2$: Generalization of Theorem~\ref{thm:batcharrival} to the case of batch services}
\\
The following theorem generalizes the main result in \cite{Hebuterne}, but is a special
case of Theorem 2 of \cite{Takine} which allows a more general arrival process (but in that theorem batch service is not considered).
\begin{theorem}
\label{thm:batchservice}
Consider the $m$-class single-node service facility of Theorem~\ref{thm:batcharrival},
with the additional assumption that
customers of class-$i$ are always served in batches of  fixed size $K_i$,
$i \in J$; the start of a service of class-$i$ customers is delayed until
$K_i$ customers are present.
Then the following relation holds between the PGF $L(\bz)$
and the PGFs $S^c_i(\bz)$,
$i \in J$:
\begin{equation}
(1 - \ee [\bz^{\bG}]) L(\bz) = \sum_{i=1}^m \frac{1-z_i^{K_i}}{K_i} \ee G_i S^c_i(\bz).
\label{eq:simple4}
\end{equation}
\end{theorem}
\begin{proof}
In view of \rem{stationary relation 1},
and after using PASTA,
Theorem \ref{thm:batchservice} is a special case of \eq{stationary relation 2} of Corollary \cort{stationary relation 1} in which there is no routing, the external arrival batch $\vc{Y}(t^{e}_{n})$ is independent of $\sr{F}_{t^{e}_{n}-}$ and the departing batch size $Y_{i}$ from queue $i$ is some constant $K_{i}$. In this case, it is easy to see that $\lambda^{d}_{i} \dd{E}[Y_{i}] = \lambda^{e}/K_{i}$, and we obtain \eq{simple4} from \eq{stationary relation 2}.
\end{proof}

\noindent
{\em Case $3$: Non-preemptive priority queues}
\\
In this example we consider a non-preemptive priority queue with $P$ customer classes.
We first verify the equality between the PGFs as given by Theorem~\ref{thm:batcharrival}
for $P=2$, and subsequently point out how one may use the theorem to obtain the steady-state joint queue-length distribution in that example for a $P$-class queue.

Consider the $M/G/1$ queue with $P$ classes of customers, with nonpreemptive priority in descending order $1,2,\dots,P$
(so Class $1$ has the highest priority).
Let $\lambda_i$ denote the arrival rate of customers of class $i$, $i=1,2$.
Takagi (\cite{Takagi}, Formula (2.87) on p. 311) presents the PGF $\Pi(z_1,z_2,\dots,z_P)$ of the steady-state joint queue-length distribution
immediately after an arbitrary customer departure epoch.
For $P=2$ he also obtains the PGF $P(z_1,z_2,\dots,z_P)$ of the steady-state joint queue-length distribution at an arbitrary epoch
(\cite{Takagi}, Formula (5.82b) on p. 397).
We have verified that, indeed, for $P=2$ classes one has (cf.\ Theorem~\ref{thm:batcharrival} with single arrivals),
\[
 (\lambda_1(1-z_1) + \lambda_2(1-z_2)) P(z_1,z_2) =
\lambda_1(1-z_1) S^c_1(z_1,z_2)
+ \lambda_2(1-z_2) S^c_2(z_1,z_2) .
\]
The starting point for this verification was the obvious set of relations,
with $\beta_i(z_1,z_2)$ the PGF of the numbers of arrivals at both queues during one service of a class-$i$ customer, $i=1,2$:
\begin{align}
\Pi_1(z_1,z_2) &:= \frac{\lambda_1}{\lambda} S^c_1(z_1,z_2) = \frac{\Pi(z_1,z_2)-\Pi(0,z_2)}{z_1} \beta_1(z_1,z_2) + \Pi(0,0) \frac{\lambda_1}{\lambda}
\beta_1(z_1,z_2),
\label{Pi1}
\\
\Pi_2(z_1,z_2) &:= \frac{\lambda_2}{\lambda} S^c_2(z_1,z_2) = \frac{\Pi(0,z_2)-\Pi(0,0)}{z_2} \beta_2(z_1,z_2) + \Pi(0,0) \frac{\lambda_2}{\lambda}
\beta_2(z_1,z_2) .
\label{Pi2}
\end{align}
Here $\Pi_i(z_1,z_2)$ is the PGF of the steady-state joint queue-length distribution immediately after
the departure of a class-$i$ customer, with indicator function $1({\rm departing ~ customer ~ is ~ of ~ class} ~ i)$, $i=1,2$,
and $\Pi(z_1,z_2)$ is as defined above.
The factors $\frac{\lambda_i}{\lambda}$ in the lefthand side of (\ref{Pi1}) and (\ref{Pi2})
are needed because the $S^c_i(z_1,z_2)$ are conditional PGFs, the condition being that the departing customer is of class $i$.

This example clearly demonstrates the value of our general balance equations. Besides providing a much shorter proof for
Takagi's Formula (5.82b), it also allows us to extend his result to the case of $P (> 2)$ customer classes,
by using the expressions for $\frac{\lambda_i}{\lambda} S^c_i(z_1,z_2,\dots,z_P)$ that follow
from Takagi's Formula (2.87) for $\Pi(z_1,z_2,\dots,z_P)$.
\\

\noindent
{\em Case $4$: Priority for the longer queue}
\\
Consider a model of one server and two queues.
Each queue has its own Poisson arrival process and service time distribution.
After a service completion, the server proceeds with a customer from the longest queue, if the queue lengths are unequal;
if the queue lengths are equal, the server chooses a customer from queue $Q_i$ with probability $\alpha_i$, $i=1,2$.
Cohen \cite{Cohen87} has derived the PGF $\Pi(z_1,z_2) = \ee[z_1^{X_1}z_2^{X_2}]$ of the steady-state joint queue-length distribution
immediately after an arbitrary customer departure epoch, by solving a Riemann-type boundary value problem.
In the process, he also obtained the following PGFs, that naturally arise in this {\em Priority for the longer queue} model:
$\ee[z_1^{X_1}z_2^{X_2} 1_{\{X_1>X_2\}}]$,
$\ee[z_1^{X_1}z_2^{X_2} 1_{\{X_1<X_2\}}]$,
and
$\ee[z_1^{X_1}z_2^{X_2} 1_{\{X_1=X_2 >0\}}]$.
Below we first show how one can obtain the PGFs $\Pi_i(z_1,z_2)$ of the steady-state joint queue-length distribution immediately after the departure
of a customer from $Q_i$, $i=1,2$ (we stick as much as possible to the notation of Case $3$).
By considering the joint queue-length distribution at two consecutive departure epochs,
and with $\beta_i(z_1,z_2)$ denoting the PGF of the numbers of arrivals at both queues during one service
of a customer from $Q_i$, we can write:
\begin{align}
\Pi_1(z_1,z_2) &=
\ee[z_1^{X_1} z_2^{X_2} 1_{\{X_1>X_2\}}] \frac{\beta_1(z_1,z_2)}{z_1}
\nonumber
\\
&+ \alpha_1
\ee[z_1^{X_1} z_2^{X_2} 1_{\{X_1=X_2 >0\}}] \frac{\beta_1(z_1,z_2)}{z_1} +
\pp(X_1=X_2=0) \frac{\lambda_1}{\lambda_1+\lambda_2}  \beta_1(z_1,z_2),
\\
\Pi_2(z_1,z_2) &=
\ee[z_1^{X_1} z_2^{X_2} 1_{\{X_1<X_2\}}] \frac{\beta_2(z_1,z_2)}{z_2}
\nonumber
\\
&+ \alpha_2
\ee[z_1^{X_1} z_2^{X_2} 1_{\{X_1=X_2 >0\}}] \frac{\beta_2(z_1,z_2)}{z_2} +
\pp(X_1=X_2=0) \frac{\lambda_2}{\lambda_1+\lambda_2}  \beta_2(z_1,z_2).
\end{align}
The queue-length PGFs in the two righthand sides are derived by Cohen \cite{Cohen87}, and thus we obtain
$\Pi_i(z_1,z_2)$, $i=1,2$.
This immediately leads to $S^c_i(z_1,z_2)$, $i=1,2$, as in Case $3$.
Subsequently, Theorem~\ref{thm:batcharrival} gives the PGF of the steady-state joint queue-length distribution at an arbitrary epoch.
It should be noticed that it is not at all easy to obtain this PGF in another way, for this non-Markovian model;
the {\em Priority for the longer queue} model is a difficult queueing model.
In the case of exponential service time distributions,
with equal arrival and service rates at the two queues and $\alpha_1=\alpha_2$,
Zheng and Zipkin \cite{ZZ} present a recursive method
to obtain this PGF, while Flatto \cite{Flatto} for this case (but allowing preemption) obtains the queue-length PGF by solving a boundary value problem.
\\

\noindent
{\em Case $5$: A simple network}
\\
Consider a network of $m$ service facilities, with independent
external Poisson arrival processes, 
and with continuous service time distributions.
We have Markovian routing, a customer moving from $Q_i$ to $Q_k$ with probability $p_{ik}$
and leaving the system after its service completion in $Q_i$ with probability $p_{i0}$, $i,k \in J$.
Define $\Lambda_i$ as the total flow through $Q_i$ per time unit, $i \in J$;
these $\Lambda_i$ are the unique solution of the set of equations
\begin{equation}
\Lambda_i = \lambda_i + \sum_{k=1}^m \Lambda_k p_{ki} , ~~~ i \in J.
\end{equation}
Let $A_i$ indicate that the system is viewed just before an arrival at $Q_i$,
$D_i$ that the system is viewed just after a departure from $Q_i$, and $I_{ik}$ that the system is viewed
just after a departure from $Q_i$ and just before the arrival of the departing customer at $Q_k$.
Letting $\vc{j} = (j_{1}, j_{2}, \ldots, j_{m})$,
one can write down the following balance equations for the queue length vector
$\vc{X} = (X_{1}, X_{2}, \ldots, X_{m})$:
\begin{align}
& \sum_{i=1}^m \lambda_i \dd{P}(\vc{X} = \vc{j}|A_i) + \sum_{i=1}^m \Lambda_i p_{i0} \dd{P}(\vc{X} = \vc{j} - \vc{e}_{i}|D_i)
+ \sum_{i=1}^m \sum_{k=1}^m \Lambda_i p_{ik} \dd{P}(\vc{X} = \vc{j} - \vc{e}_{i}|I_{ik}) \nonumber
\\
&=
\sum_{i=1}^m \Lambda_i p_{i0} \dd{P}(\vc{X} = \vc{j}|D_i)
+\sum_{i=1}^m \lambda_i \dd{P}(\vc{X} = \vc{j} - \vc{e}_{i} |A_i)
+
\sum_{i=1}^m \sum_{k=1}^m \Lambda_i p_{ik} \dd{P}(\vc{X} = \vc{j} - \vc{e}_{k}|I_{ik}).
\label{eq:balnetwork}
\end{align}

The (PGF of the) probabilities,
given that we observe just after a real departure from $Q_i$ or that we observe
just after a departure from $Q_i$ that will in an instant result in an arrival at $Q_k$,
are obviously the same.
If one takes PGFs, one quickly sees that a special case of \eq{stationary relation 5} is obtained.

We now use \eq{balnetwork} to provide an alternative proof for the joint queue-length distribution in a queueing network with a single roving server as studied in \cite{BMW,SLF}.
Again consider a network of $m$ queues with Markovian customer routing, as described above.
In this particular example, we assume that a \emph{single} server visits the queues in a fixed, cyclic order, requiring a switch-over time $S_i$ to move from $Q_i$ to $Q_{i+1}$.
We do not make any assumptions regarding the service disciplines at each queue. This model, which can be regarded as a polling model with customer routing, has been studied by Sidi, Levy and Fuhrmann \cite{SLF} who refer to this model as a queueing network with a roving server. Sidi et al. obtain the joint queue-length distribution at arbitrary moments, as well as the joint queue-length distribution at departure epochs. The waiting-time distributions are obtained in a different paper \cite{BMW}. For us, it is slightly more convenient to refer to this latter paper in the analysis below, because the authors in \cite{BMW} use the same definition of $V^c_i(\bz)$, the PGF of the joint queue length at departure epochs, just after a departure from $Q_i$ and just \emph{before} the arrival of the departing customer at the next queue.

Take the formulas (3.2)--(3.6) of \cite{BMW}. From (3.2), which is the counterpart of our (\ref{eq:equilibrium}), one can express (in the notation of the present paper)
the differences of PGFs at visit beginning and visit completion epochs into those at service beginning
and service completion epochs:
\begin{equation}
\label{eq:equilibriuma}
\frac{ V^b_i(\bz) - V^c_i(\bz)}{\Lambda_i\ee C} =  S^b_i(\bz) - S^c_i(\bz) P_i(\bz), \quad i=1,2,\dots,m.
\end{equation}
Here $P_i(\bz) := p_{i0} + \sum_{k=1}^m p_{ik} z_k$, and $\ee C=s/(1-\rho)$ with $\rho:=\sum_{i=1}^m\Lambda_i b_i$.
Next use our relation (\ref{eq:relation}) to express $S^b_i(\bz)$ into $S^c_i(\bz)$.
Subsequently express $L(\bz)$, in (3.4) of \cite{BMW}, which is the counterpart of (\ref{eq:QL}) above,
in differences $V^b_i(\bz) - V^c_i(\bz)$, as was also done in \cite{BKK}.
This gives
\begin{equation}
\sum_{i=1}^m \lambda_i(1-z_i) L(\bz) = \sum_{i=1}^m \Lambda_i (P_i(\bz) - z_i) S^c_i(\bz).
\label{eq:roving}
\end{equation}
This is indeed in agreement with \eq{balnetwork}:
the LHS of (\ref{eq:roving}) gives the first and the fifth term in~\eq{balnetwork}.
The last term in the RHS gives the second plus the third term in \eq{balnetwork},
once we realize that $p_{i0}+\sum_{k=1}^m p_{ik} = 1$, and that the conditional
probabilities both refer to a service completion in $Q_i$, no matter whether the condition is $D_i$ or $I_{ik}$.
The first term in the RHS gives the fourth plus fifth term in \eq{balnetwork}.
One could argue that some results in \cite{BMW} and \cite{SLF} could have been derived faster by starting from
(\ref{eq:roving}).

\section{Concluding remarks}
\label{sect:concluding}

This paper derives a distributional relationship, at different embedded epochs, for analyzing queues and their networks. As shown in \sectn{formal}, it has different forms according to the abstraction level of the model.
This may both lead to new results and easier derivations of some known results.
In Section~\ref{sec:3} this is demonstrated for a few examples.

The relationship in \sectn{non stationary} has a different nature than the rest of this paper because it does not require any stationarity of the processes of interest. Namely, it suggests that such an asymptotic relationship may enable us to obtain queueing characteristics with some error bounds, not assuming any stationarity condition. This is completely different from the standard analysis in queueing theory. Thus, it would be interesting to see whether it can yield useful results for the performance evaluation of queueing models. We leave this for future studies.

\appendix
\normalsize
\section*{Appendix}

In the appendices below, we omit ``a.s.'' because countably many events each of which occurs w.p. 1 simultaneously occur w.p. 1.

\section{Proof of \lem{stationary 1}}
\label{app:stationary 1}

Since the proofs of \eq{stationary f1} and \eq{stationary g1} are similar, we only prove \eq{stationary f1}. Since $\pi^{e}$ is a probability distribution, we can choose a sufficiently large $a$ for each $\epsilon > 0$ such that
\begin{align*}
  \sum_{\max(|\vc{x}|,|\vc{y}|) \ge a} \pi^{e}(\vc{x},\vc{y}) < \epsilon.
\end{align*}
Let $\sr{S}_{a} = \{(\vc{x},\vc{y}) \in \dd{Z}_{+}^{2m}; max(|\vc{x}|,|\vc{y}|) < a\}$, then $\sr{S}_{a}$ is a finite set. Hence, summing both sides of \eq{stationary 1} for $(\vc{x},\vc{y}) \in \sr{S}_{a}$ yields
\begin{align*}
 \lim_{n \to \infty} \frac 1n \sum_{\ell=1}^{n} \sum_{(\vc{x},\vc{y}) \in \sr{S}_{a}} 1(\vc{X}(t^{e}_{\ell}-) = \vc{x}, \Delta \vc{N}^{e}(t^{e}_{\ell}) = \vc{y}) = \sum_{(\vc{x},\vc{y}) \in \sr{S}_{a}} \pi^{e}(\vc{x},\vc{y}),
\end{align*}
and therefore
\begin{align}
\label{eq:tail 1}
  \lim_{n \to \infty} \frac 1n \sum_{\ell=1}^{n} \sum_{(\vc{x},\vc{y}) \not\in \sr{S}_{a}} & 1(\vc{X}(t^{e}_{\ell}-) = \vc{x}, \Delta \vc{N}^{e}(t^{e}_{\ell}) = \vc{y}) \nonumber \\
  & = 1 - \lim_{n \to \infty} \frac 1n \sum_{\ell=1}^{n} \sum_{(\vc{x},\vc{y}) \in \sr{S}_{a}} 1(\vc{X}(t^{e}_{\ell}-) = \vc{x}, \Delta \vc{N}^{e}(t^{e}_{\ell}) = \vc{y}) \nonumber \\
  & = 1 - \sum_{(\vc{x},\vc{y}) \in \sr{S}_{a}} \pi^{e}(\vc{x},\vc{y}) = \sum_{\max(|\vc{x}|,|\vc{y}|) \ge a} \pi^{e}(\vc{x},\vc{y}) < \epsilon.
\end{align}
Multiplying both sides of \eq{stationary 1} by $g(\vc{x},\vc{y})$ and summing them for $(\vc{x},\vc{y}) \in \sr{S}_{a}$ yields
\begin{align*}
 \lim_{n \to \infty} \frac 1n \sum_{\ell=1}^{n} \sum_{(\vc{x},\vc{y}) \in \sr{S}_{a}} g(\vc{x},\vc{y}) 1(\vc{X}(t^{e}_{\ell}-) = \vc{x}, \Delta \vc{N}^{e}(t^{e}_{\ell}) = \vc{y}) = \sum_{(\vc{x},\vc{y}) \in \sr{S}_{a}} g(\vc{x},\vc{y}) \pi^{e}(\vc{x},\vc{y}).
\end{align*}
Let $\|g\| = \sup_{\vc{x},\vc{y}} g(\vc{x},\vc{y})$, which is finite by the assumption. Since \eq{tail 1} implies that
\begin{align*}
 & \limsup_{n \to \infty} \frac 1n \sum_{\ell=1}^{n} \sum_{(\vc{x},\vc{y}) \not\in \sr{S}_{a}} g(\vc{x},\vc{y}) 1(\vc{X}(t^{e}_{\ell}-) = \vc{x}, \Delta \vc{N}^{e}(t^{e}_{\ell}) = \vc{y}) \\
 & \hspace{5ex} \le \|g\| \limsup_{n \to \infty} \frac 1n \sum_{\ell=1}^{n} \sum_{(\vc{x},\vc{y}) \not\in \sr{S}_{a}} 1(\vc{X}(t^{e}_{\ell}-) = \vc{x}, \Delta \vc{N}^{e}(t^{e}_{\ell}) = \vc{y}) < \|g\| \epsilon,\\
 & \sum_{(\vc{x},\vc{y}) \not\in \sr{S}_{a}} g(\vc{x},\vc{y}) \pi^{e}(\vc{x},\vc{y}) < \|g\| \epsilon,
\end{align*}
we have
\begin{align*}
 \limsup_{n \to \infty} \Big|\frac 1n \sum_{\ell=1}^{n} \sum_{\vc{x},\vc{y}} g(\vc{x},\vc{y}) 1(\vc{X}(t^{e}_{\ell}-) = \vc{x}, \Delta \vc{N}^{e}(t^{e}_{\ell}) = \vc{y}) - \sum_{\vc{x},\vc{y}} g(\vc{x},\vc{y}) \pi^{e}(\vc{x},\vc{y})\Big| < 2 \|g\| \epsilon.
\end{align*}
Letting $\epsilon \downarrow 0$, we arrive at \eq{stationary f1}.

\section{Proof of \lem{single 1}}
\label{app:single 1}

In view of \lem{stationary 1}, to prove (\ref{eq:stationary local 1}) it suffices to prove that, for $A \in 2^{J} \setminus \{\emptyset\}$,
\begin{align}
\label{eq:stationary local 3}
 & \lim_{n \to \infty} \frac 1n \sum_{\ell=1}^{n} 1(\vc{X}(t^{d}_{A,\ell}-) = \vc{x}, \Delta \vc{N}^{d}(t^{d}_{A,\ell})= \vc{y}, \Delta \vc{N}^{r}(t^{d}_{A,\ell}) = \vc{z}) = \pi^{d}_{A}(\vc{x},\vc{y},\vc{z}).
\end{align}
It follows from v) that, for each $i \in J, \ell \ge 1, \vc{y} \in S_{A}, \vc{z} \in \dd{Z}_+^m$, there is a unique $k \ge 1$ such that $\ell \le k$ and
\begin{align*}
  1(\vc{X}(t^{d}_{A,\ell}-) = \vc{x}, & \Delta \vc{N}^{d}(t^{d}_{A,\ell}) = \vc{y}, \Delta \vc{N}^{d}(t^{d}_{A,\ell}) = \vc{z}) \nonumber\\
 & = 1(\vc{X}(t^{d}_{k}-) = \vc{x}, \Delta \vc{N}^{d}(t^{d}_{k}) = \vc{y}, \Delta \vc{N}^{d}(t^{d}_{k}) = \vc{z}),
\end{align*}
and \eq{rate 1} and \eq{rate 2} imply
\begin{align*}
  \lim_{t \to \infty} \frac {\tilde{N}^{d}_{A}(t)} {|\tilde{\vc{N}}^{d}|(t)} = \lim_{t \to \infty} \frac {\frac 1t \tilde{N}^{d}_{A}(t)} {\frac 1t |\tilde{\vc{N}}^{d}|(t)} = \frac {\lambda^{d}_{A}} {\lambda^{d}}.
\end{align*}
Hence, for $\vc{y} \in S_{A}$,
\begin{align*}
 \pi^{d}(\vc{x},\vc{y},\vc{z}) & = \lim_{n \to \infty} \frac 1n \sum_{k=1}^{n} 1(\vc{X}(t^{d}_{k}-) = \vc{x}, \Delta \vc{N}^{d}(t^{d}_{k}) = \vc{y},\Delta \vc{N}^{r}(t^{d}_{k}) = \vc{z}) \\
 & = \lim_{t \to \infty} \frac 1{|\tilde{\vc{N}}^{d}|(t)} \sum_{k=1}^{|\tilde{\vc{N}}^{d}|(t)} 1(\vc{X}(t^{d}_{k}-) = \vc{x}, \Delta \vc{N}^{d}(t^{d}_{k}) = \vc{y}, \Delta \vc{N}^{r}(t^{d}_{k}) = \vc{z}) \\
 & = \lim_{t \to \infty} \frac {\tilde{N}^{d}_{A}(t)}{|\tilde{\vc{N}}^{d}|(t)} \frac 1{\tilde{N}^{d}_{A}(t)} \sum_{\ell=1}^{\tilde{N}^{d}_A(t)} 1(\vc{X}(t^{d}_{A,\ell}-) = \vc{x}, \Delta \vc{N}^{d}(t^{d}_{A,\ell}) = \vc{y},\Delta \vc{N}^{r}(t^{d}_{A,\ell}) = \vc{z})\\
 & = \frac {\lambda^{d}_{A}} {\lambda^{d}} \lim_{n \to \infty}\frac 1n \sum_{\ell=1}^{n} 1(\vc{X}(t^{d}_{A,\ell}-) = \vc{x}, \Delta \vc{N}^{d}(t^{d}_{A,\ell}) = \vc{y},\Delta \vc{N}^{r}(t^{d}_{A,\ell}) = \vc{z}).
\end{align*}
This proves \eq{stationary local 3} by the definition $\pi^{d}_{A}$, and therefore (\ref{eq:stationary local 1}) holds. The fact that $\pi^{d}_{A}$ is a probability distribution is immediate from \eq{stationary local 1} with $h(\vc{x},\vc{y},\vc{z}) \equiv 1$.
\\

\noindent
{\bf Acknowledgment}. We are grateful to a referee for providing useful references and insightful remarks.


\begin{thebibliography}{99}
\bibitem{BaccBrem2003}
F. Baccelli and P. Br{\'e}maud (2003).
\newblock \textit{Elements of queueing theory: Palm martingale calculus and
  stochastic recurrences}, vol.~26 of \textit{Applications of Mathematics}.
\newblock 2nd ed. Springer, Berlin.

\bibitem{BB}
S.C. Borst and O.J. Boxma (1997).
Polling models with and without switchover times.
{\em Oper. Res.} {\bf 45}, 536-543.

\bibitem{BKK}
O.J. Boxma, O. Kella and K.M. Kosinski (2011).
Queue lengths and workloads in polling systems.
{\em Oper. Res. Letters} {\bf 39}, 401-405.

\bibitem{BMW}
M.A.A. Boon, R.D. van der Mei and E.M.M. Winands (2013).
Waiting times in queueing networks with a single shared server.
{\em Queueing Systems} {\bf 74}, 403-429.

\bibitem{BT}
O.J. Boxma and T. Takine (2003).
The $M/G/1$ FIFO queue with several customer classes.
{\em Queueing Systems} {\bf 45}, 185-189.

\bibitem{Cohen87}
J.W. Cohen (1987).
A two-queue, one-server model with priority for the longer queue.
{\em Queueing Systems} {\bf 2}, 261-283.

\bibitem{Cooper}
R.B. Cooper (1972).
{\em Introduction to Queueing Theory}.
Macmillan, New York.

\bibitem{Eisenberg}
M. Eisenberg (1972).
Queues with periodic service and changeover time.
{\em Oper. Res.} {\bf 20}, 440-451.

\bibitem{Fakinos}
D. Fakinos (1991).
The relation between limiting queue size distributions at arrival and departure epochs in a bulk queue.
{\em Stochastic Processes and their Applications} {\bf 37}, 327-329.

\bibitem{Flatto}
L. Flatto (1989).
The longer queue model.
{\em Prob. Eng. Inform. Sci.} {\bf 3}, 537-559.

%
%
\bibitem{Hebuterne}
G. H\'ebuterne (1988).
Relation between states observed by arriving and departing customers in bulk systems.
{\em Stochastic Processes and their Applications} {\bf 27}, 279-289.

\bibitem{HebuterneRosenberg}
G. H\'ebuterne and C. Rosenberg (1999).
Arrival and departure state distributions in the general bulk-service queue.
{\em Naval Research Logistics Quarterly} {\bf 46}, 107-118.

\bibitem{Kim}
K. Kim (2015),
A relation between queue-length distributions during server vacations in queues with batch
arrivals, batch services, or multiclass arrivals: An extension of Burke's theorem.
{\em Indian Journal of Science and Technology} {\bf 8}, 1-5.

\bibitem{Miya1994}
M. Miyazawa (1994) Rate conservation laws: a survey.
{\em Queueing Systems} {\bf 15}, 1-58.

\bibitem{Miya2010}
M. Miyazawa (2010).
\newblock \textit{Palm calculus, reallocatable GSMP and insensitivity
  structure}, chap. 4 of Queueing networks: A fundamental approach.
\newblock International Series in Operations Research and Management Science,
  Springer, 141--215.

\bibitem{PapaBertsimas}
X. Papaconstantinou and D. Bertsimas (1990).
Relations between the pre-arrival and the post-departure state probabilities and the FCFS waiting time distribution
in the $E_k/G/s$ queue.
{\em Naval Research Logistics} {\bf 37}, 135-149.


\bibitem{SLF}
M. Sidi, H. Levy and S.W. Fuhrmann (1992).
A queueing network with a single cyclically roving server.
{\em Queueing Systems} {\bf 11}, 121-144.

\bibitem{Takagi}
H. Takagi (1991).
Queueing Analysis.
A foundation of performance evaluation.
Volume 1: Vacation and priority systems.
{\em North-Holland Publ. Cy., Amsterdam}.

\bibitem{Takine}
T. Takine (2001).
Distributional form of Little's law for FIFO queues with multiple Markovian arrival streams and its applications to queues
with vacations.
{\em Queueing Systems} {\bf 37}, 31-63.

\bibitem{ZZ}
Y.-S. Zheng and P. Zipkin (1990).
A queueing model to analyze the value of centralized inventory information.
{\em Oper. Res.} {\bf 38}, 296-307.

\end{thebibliography}
\end{document}